\documentclass[10pt,a4paper,oneside]{article}
\usepackage[english]{babel}
\usepackage{a4wide,amsmath,amsthm,amssymb,url,graphicx,xspace,algorithm,algorithmic,pgf}
\usepackage[latin1]{inputenc}
\usepackage{enumitem}
\usepackage{multirow}
\usepackage{hyperref}
\usepackage{tabularx}
\usepackage{tikz}
\usepackage{allrunes}
\usepackage{bm}

\def\F{\mathbb{F}}

\def\q{\mathcal{Q}}
\def\w{\mathcal{W}}
\def\h{\mathcal{H}}

\DeclareMathOperator{\PG}{PG}

\theoremstyle{definition}
\newtheorem{theorem}{Theorem}[section]
\newtheorem{lemma}[theorem]{Lemma}
\newtheorem{definition}[theorem]{Definition}
\newtheorem{remark}[theorem]{Remark}
\newtheorem{corollary}[theorem]{Corollary}

\newcommand{\comments}[1]{}
\newcommand{\gs}[3]{\genfrac{[}{]}{0pt}{0}{#1}{#2}_{#3}}

\author{Jan De Beule\footnote{Address: Vrije Universiteit Brussel, Department of Mathematics, Pleinlaan 2, B--1050 Brussel, Belgium\newline Email address: jan@debeule.eu Website: www.debeule.eu}~ and Maarten De Boeck\footnote{Address: UGent, Department of Mathematics, Krijgslaan 281 -- S25, 9000 Gent, Flanders. \newline Email address: Maarten.DeBoeck@UGent.be}}
\title{A combinatorial characterisation of embedded polar spaces}
\date{}
\begin{document}
\maketitle

\begin{abstract}
Some classical polar spaces admit polar spaces of the same rank as embedded polar spaces (often arisen as the intersection of the polar space with a non-tangent hyperplane). In this article we look at sets of generators that behave combinatorially as the set of generators of such an embedded polar space, and we prove that they are the set of generators of an embedded polar space.
\end{abstract}


\section{Introduction}\label{sec:introduction}

A fundamental question in finite geometry is to recognize geometric substructures from combinatorial properties. One of the first questions of this kind was posed (and solved) by Beniamino Segre, who showed that an {\em oval} (which is defined in a combinatorial way) of a finite Desarguesian projective plane of odd order, is necessarily a {\em conic} (which is a non-singular curve of degree two). Comparable questions have been studied in higher dimensional projective spaces and also in finite classical polar spaces, of which the following theorem is an example. 

\begin{theorem}[De Winter and Schillewaert, \cite{DWS2010}]
If a point set $K$ in $\PG(n,q), n \geq 4, q > 2$, has the same intersection numbers with respect to hyperplanes and subspaces of codimension 2 as a polar space $\mathcal{P} \in \{\h(n,q^2),\q^+(2n+1,q),\q^-(2n+1,q),\q(2n,q)\}$, then $K$ is the point set of a non-singular polar space $\mathcal{P}$.
\end{theorem}

In this paper, we deal with a rather particular situation in finite classical polar spaces. Motivated by research in \cite{dss2}, the aim is to recognize embedded finite classical polar spaces as sets of generators of a larger polar space satisfying some combinatorial properties. As such, we want to provide a proof of Theorem 6.6 in \cite{dss2}, and this is the main aim of this paper.

Polar spaces were introduced in an axiomatic way by Veldkamp (\cite{vel1,vel2}) and Tits (\cite{tits}). 
A {\em polar space} is a point-line geometry satisfying the one-or-all axiom, i.e. for a given point $P$ and a given line $l$ not through $P$, the point $P$ is collinear
with one point of $l$ or with all points of $l$. This characterization is due to Buekenhout and Shult. Although this remarkable characterization is often very useful 
in geometrical and combinatorial proofs of theorems on polar spaces, we will prefer the use of the original definition of polar spaces of Tits, which turns out to
make our proofs shorter.

\begin{definition}\label{defpolarspace}
A \index{polar space}\emph{polar space of rank $d$}, $d\geq3$, is an incidence geometry $(\Pi,\Omega)$ with $\Pi$ a set whose elements are called points and $\Omega$ a set of subsets of $\Pi$ satisfying the following axioms.
\begin{enumerate}
\item Any element $\omega\in\Omega$ together with the elements of $\Omega$ that are contained in 
$\omega$, is a projective geometry of dimension at most $d-1$. 
\item The intersection of two elements of $\Omega$ is an element of $\Omega$ 
(the set $\Omega$ is closed under intersections).
\item For a point $P\in\Pi$ and an element $\omega\in\Omega$ of dimension $d-1$ 
such that $P$ is not contained in $\omega$ there is a unique element $\omega'\in\Omega$ of dimension $d-1$ 
containing $P$ such that $\omega\cap\omega'$ is a hyperplane of $\omega$. The element $\omega$ is the 
union of all 1-dimensional elements of $\Omega$ that contain $P$ and are contained in $\omega$.
\item There exist two elements $\Omega$ both of dimension $d-1$ whose intersection is empty.
\end{enumerate}
\end{definition}

One of the consequences of the theory developed in \cite{tits} is that all polar spaces of rank at least $3$ 
arise from a sesquilinear or quadratic form acting on a vector space over a (skew)field. In the finite case, 
this means that finite polar spaces of rank at least $3$ are known and classified. We assume that the reader is familiar with 
finite classical polar spaces. To fix the notation, we provide a table with the six different families of finite classical
polar spaces including rank and parameter.

The finite field of order $q$, $q=p^h$, $p$ prime and $h \geq 1$, will be denoted as $\F_q$, and the $n$-dimensional
projective space over $\F_q$ as $\PG(n,q)$.

\begin{table}[ht]
\begin{center}
\begin{tabular}{|c|c|c|c|c|c|}
\hline
family & subfamily & notation & ambient space & rank & parameter \\
\hline\hline
\multirow{3}{*}{orthogonal} & elliptic & $\q^-(2n+1,q)$ & $\PG(2n+1,q), n \geq 1$ & $n$ & 2\\
& parabolic & $\q(2n,q)$ & $\PG(2n,q), n \geq 1$ & $n$ & 1\\ 
& hyperbolic & $\q^+(2n+1,q)$ & $\PG(2n+1,q), n \geq 0$ & $n+1$ & 0\\
\hline
\multirow{2}{*}{Hermitian} & odd dimension & $\h(2n+1,q^2)$ & $\PG(2n+1,q^2), n \geq 0$ & $n+1$ & $1/2$ \\
& even dimension & $\h(2n,q^2)$ & $\PG(2n,q^2)$, $n \geq 1$ & $n$ & 3/2 \\
\hline
symplectic & & $\w(2n+1,q)$ & $\PG(2n+1,q)$, $n \geq 0$ & $n+1$ & 1\\
\hline
\end{tabular}
\end{center}
\caption{Different flavors of finite classical polar spaces, parameter and rank}\label{tab:flavor}
\end{table}

A finite classical polar space of rank $d$ over $\F_{q}$ has parameter $e=\log_{q}(x-1)$ with $x$ the 
number of generators through a fixed $(d-2)$-space.  The following lemma summarizes the number
of points and generators of a finite classical polar space using rank, parameter and order of the field.
Recall that $\gs{n+1}{k+1}{q}$ denotes the Gaussian coefficient representing 
the number of $k$-dimensional spaces in $\PG(n,q)$. Note that for $0 \leq m < r$ 
one defines $\gs{m}{r}{q} = 0$. 

\begin{lemma}\label{subspacesonpolar}%\label{subspacethroughonpolar}
	The number of generators of a finite classical polar space of rank $d$ with parameter $e$, embedded in a projective space over $\F_{q}$, is given by $\prod^{d-1}_{i=0}(q^{e+i}+1)$. Its number of points equals $\gs{d}{1}{q}(q^{d+e-1}+1)$.
	\par The number of generators on a classical finite polar space of rank $d$ with parameter $e$, embedded in a projective space over $\F_{q}$, through a fixed point is $\prod^{d-2}_{i=0}(q^{e+i}+1)$.%	\gs{d-m-1}{k-m}{q}\prod^{k-m}_{i=1}(q^{d+e-m-i-1}+1).
\end{lemma}

%Note that the number of generators through a fixed point is given by . %Note that the number of generators through a fixed line is given by $\prod^{d-3}_{i=0}(q^{e+i}+1)$.

Consider a polar space $\mathcal{P}$ of rank $d$, defined over $\F_{q}$. Any hyperplane $\pi$ of the ambient projective space which is not a tangent hyperplane to $\mathcal{P}$, contains or intersects the elements of $\mathcal{P}$. The elements completely contained in the hyperplane constitute a finite classical polar space $\mathcal{P}'$ in $\pi$. The polar space $\mathcal{P}'$ may be of the same rank as $\mathcal{P}$, but will have a different parameter. In this paper we are interested in the cases where the rank of $\mathcal{P}'$ equals the rank of $\mathcal{P}$. This restricts us to the following cases: $\q^+(2n-1,q) \subset \q(2n,q)$, $\q(2n,q) \subset \q^-(2n+1,q)$ and $\h(2n-1,q^2) \subset \h(2n,q^2)$.

Using the particular isomorphism between $\q(2n,q)$ and $\w(2n-1,q)$, $q$ even, also the embedding $\q^{+}(2n+1,q)\subset\w(2n+1,q)$, $q$ even, is known. Our result will also include this case.

\begin{definition}\label{defgood}
Let $\mathcal{P}$ be a finite classical polar space of rank $d\geq3$ and with parameter $e\geq1$, 
embedded in a projective space over $\F_{q}$. A set $\mathcal{S}$ of generators of $\mathcal{P}$ is called \emph{strong pseudopolar} if
\begin{enumerate}[label=(\roman*)]
\item for every $i=0,\dots,d$ the number of elements of $\mathcal{S}$ meeting a generator $\pi$ in a $(d-i-1)$-space equals
\[
\begin{cases}
\left(\gs{d-1}{i-1}{q}+q^{i}\gs{d-1}{i}{q}\right)q^{\binom{i-1}{2}+ie-1}&\text{if }\pi\in\mathcal{S}\\
(q^{e-1}+1)\gs{d-1}{i-1}{q}q^{\binom{i-1}{2}+(i-1)e}&\text{if }\pi\notin\mathcal{S}
\end{cases}\;;
\]
\item for every point $P$ of $\mathcal{P}$ there is a generator $\pi\notin\mathcal{S}$ through $P$;
\item for every point $P$ of $\mathcal{P}$ and every generator $\pi\notin\mathcal{S}$ through $P$, 
there are either $(q^{e-1}+1)\gs{d-2}{j}{q}q^{\binom{j}{2}+je}$ generators of $\mathcal{L}$ through $P$ meeting 
$\tau$ in a $(d-j-2)$-space, for all $j=0,\dots, d-2$, or there are no generators of $\mathcal{L}$ through $P$ 
meeting $\tau$ in a $(d-j-2)$-space, for all $j=0,\dots, d-2$.
\end{enumerate}
\end{definition}

The aim of this paper is precisely to show a characterisation of polar spaces $\mathcal{P}'$ 
embedded in a polar space $\mathcal{P}$ of the same rank through the combinatorial and 
geometrical behavior of the set of generators of $\mathcal{P}'$ as subset of the generators 
of $\mathcal{P}$. In other words, if a set of generators of a finite classical polar space behaves 
combinatorially as the set of generators of an embedded polar space of the same rank, is it 
really the set of generators of an embedded polar space? The main theorem of this paper is the following.

\begin{theorem}
Let $\mathcal{P}$ be a finite classical polar space of rank $d$ and with parameter $e\geq1$ over the finite field $\F_q$. 
If $\mathcal{S}$ is a strong pseudopolar set of generators in $\mathcal{P}$, then $\mathcal{S}$ is the set of generators of a 
finite classical polar space of rank $d$ and with parameter $e-1$ embedded in $\mathcal{P}$. 
\end{theorem}

\section{The main theorem}

We start by presenting a well-known result for Gaussian coefficients, the \emph{$q$-binomial theorem}. It is a $q$-analogue of the classical binomial theorem.

\begin{lemma}\label{qbinomial}
	For any prime power $q$, non-negative integer $n$ and $t$ a variable over $\F_{q}$ we have
	\[
	\prod^{n-1}_{l=0}(1+q^{l}t)=\sum^{n}_{l=0}q^{\binom{l}{2}}\gs{n}{l}{q}t^{l}\;.
	\]
\end{lemma}

First we present two counting results for strong pseudopolar sets.

\begin{lemma}\label{aantallen}
	Let $\mathcal{P}$ be a finite classical polar space of rank $d\geq3$ and with parameter $e\geq1$, embedded in a projective space over $\F_{q}$ and let $\mathcal{S}$ be a strong pseudopolar set of generators in $\mathcal{P}$.
	\begin{enumerate}[label=(\roman*)]
		\item $|\mathcal{S}|=\prod_{i=0}^{d-1}(q^{e+i-1}+1)$.
		\item Any point of $\mathcal{P}$ is contained in $0$ or $\prod_{i=0}^{d-2}(q^{e+i-1}+1)$ elements of $\mathcal{S}$.
	\end{enumerate}
\end{lemma}
\begin{proof}
	\begin{enumerate}[label=(\roman*)]
		\item By condition (ii) of strong pseudopolar sets it is clear that $\mathcal{S}$ cannot contain all generators of $\mathcal{P}$. Let $\tau$ be a generator of $\mathcal{P}$ not contained in $\mathcal{S}$. Counting the elements of $\mathcal{S}$ by intersection dimension with $\tau$, using condition (i) we find that
		\begin{align*}
		|\mathcal{S}|=\sum_{i=1}^{d}(q^{e-1}+1)\gs{d-1}{i-1}{q}q^{\binom{i-1}{2}+(i-1)e}=(q^{e-1}+1)\sum_{i=0}^{d-1}\gs{d-1}{i}{q}q^{\binom{i}{2}}(q^{e})^{i}=\prod_{i=0}^{d-1}(q^{e+i-1}+1)\;.
		\end{align*}
		\item Let $P$ be a point of $\mathcal{P}$. By condition (ii) of strong pseudopolar sets we find a generator containing $P$ not in $\mathcal{S}$. The lemma now follows from condition (iii) since
		\begin{align*}
		\sum_{j=0}^{d-2}(q^{e-1}+1)\gs{d-2}{j}{q}q^{\binom{j}{2}+je}=(q^{e-1}+1)\sum_{j=0}^{d-2}\gs{d-2}{j}{q}q^{\binom{j}{2}}(q^{e})^{j}=(q^{e-1}+1)\prod_{i=0}^{d-3}(q^{e+i}+1)\;.
		\end{align*}
	\end{enumerate}
	In both cases we used Lemma \ref{qbinomial} in the final step.
\end{proof}

\comments{\color{blue}
\begin{lemma}\label{aantal}
	Let $\mathcal{P}$ be a finite classical polar space of rank $d\geq3$ and with parameter $e\geq1$, embedded in a projective space over $\F_{q}$. A good set of generators in $\mathcal{P}$ contains $\prod_{i=0}^{d-1}(q^{e+i-1}+1)$ generators.
\end{lemma}
\begin{proof}
	By applying condition (ii) of good sets it is clear that $\mathcal{S}$ cannot contain all generators of $\mathcal{P}$. Let $\tau$ be a generator of $\mathcal{P}$ not contained in $\mathcal{S}$. Counting the elements of $\mathcal{S}$ by intersection dimension with $\tau$, using condition (i) we find that
	\begin{align*}
	|\mathcal{S}|=\sum_{i=1}^{d}(q^{e-1}+1)\gs{d-1}{i-1}{q}q^{\binom{i-1}{2}+(i-1)e}=(q^{e-1}+1)\sum_{i=0}^{d-1}\gs{d-1}{i}{q}q^{\binom{i}{2}}(q^{e})^{i}=\prod_{i=0}^{d-1}(q^{e+i-1}+1)\;.
	\end{align*}
	In the final step we used Lemma \ref{qbinomial}.
\end{proof}

\begin{lemma}\label{aantaldoorpunt}
	Let $\mathcal{P}$ be a finite classical polar space of rank $d\geq3$ and with parameter $e\geq1$, embedded in a projective space over $\F_{q}$ and let $\mathcal{S}$ be a good set of generators in $\mathcal{P}$. 
\end{lemma}
\begin{proof}
	Let $P$ be a point of $\mathcal{P}$. By condition (ii) of good sets we find a generator containing $P$ not in $\mathcal{S}$. The lemma now follows from condition (iii) since
	\[
	\sum_{j=0}^{d-2}(q^{e-1}+1)\gs{d-2}{j}{q}q^{\binom{j}{2}+je}=(q^{e-1}+1)\sum_{j=0}^{d-2}\gs{d-2}{j}{q}q^{\binom{j}{2}}(q^{e})^{j}=(q^{e-1}+1)\prod_{i=0}^{d-3}(q^{e+i}+1)\;.
	\]
	In the final step we used Lemma \ref{qbinomial}.
\end{proof}
}

Now, we define pseudopolar sets of generators, with similar but different characterisations.

\begin{definition}\label{defnice}
	Let $\mathcal{P}$ be a finite classical polar space of rank $d\geq3$ and with parameter $e\geq1$, embedded in a projective space over $\F_{q}$. A set $\mathcal{S}$ of generators of $\mathcal{P}$ is called \emph{pseudopolar} if
	\begin{itemize}
		\item[(i)] the number of elements of $\mathcal{S}$ meeting a given generator of $\mathcal{S}$ in a $(d-2)$-space equals $\gs{d}{1}{q}q^{e-1}$;
		\item[(ii)] the number of elements of $\mathcal{S}$ disjoint to a given generator of $\mathcal{S}$ is nonzero;
		\item[(iii)] $|\mathcal{S}|=\prod_{i=0}^{d-1}(q^{e+i-1}+1)$;
		\item[(iv)] any point of $\mathcal{P}$ is contained in $0$ or $\prod_{i=0}^{d-2}(q^{e+i-1}+1)$ elements of $\mathcal{S}$.
	\end{itemize}
\end{definition}

\comments{
\begin{definition}\label{defkind}
	Let $\mathcal{P}$ be a finite classical polar space of rank $d\geq3$ and with parameter $e\geq1$, embedded in a projective space over $\F_{q}$. A set $\mathcal{S}$ of generators of $\mathcal{P}$ is called \emph{kind} if
	\begin{itemize}
		 \item[(i)] the number of elements of $\mathcal{S}$ meeting a given generator of $\mathcal{S}$ in a $(d-2)$-space equals $\gs{d}{1}{q}q^{e-1}$;
		 \item[(ii)] the number of elements of $\mathcal{S}$ disjoint to a given generator of $\mathcal{S}$ is nonzero;
		 \item[(iii)] the number of elements of $\mathcal{S}$ meeting a given generator not in $\mathcal{S}$ in a $(d-2)$-space equals $q^{e-1}+1$;
		 \item[(iv)] for every point $P$ of $\mathcal{P}$ and every generator $\pi\notin\mathcal{S}$ through $P$, either for all $j=0,\dots,d-2$ the number of generators of $\mathcal{S}$ through $P$ meeting $\tau$ in a $j$-space is non-zero and it equals $q^{e+1}+1$ if $j=d-2$, or else for all $j=0,\dots,d-2$ it is zero.
	\end{itemize}
\end{definition}}

For a set of generators being \comments{kind or }pseudopolar is a consequence of being strong pseudopolar. We will show this in detail.

\begin{lemma}\label{goodnicekind}
	Let $\mathcal{P}$ be a finite classical polar space of rank $d\geq3$ and with parameter $e\geq1$, embedded in a projective space over $\F_{q}$. If $\mathcal{S}$ is a strong pseudopolar set of generators in $\mathcal{P}$, then $\mathcal{S}$ is also a pseudopolar \comments{and a kind }set of generators.
\end{lemma}
\begin{proof}
	The first condition for \comments{both kind and }pseudopolar sets follows from condition (i) of strong pseudopolar sets for $i=1$ and $\pi\in\mathcal{S}$. The second condition for \comments{both kind and }pseudopolar sets also follows from condition (i) of strong pseudopolar sets, now applied for $i=d$ and $\pi\in\mathcal{S}$: we find that there are $q^{\binom{d-1}{2}+de-1}>0$ generators in $\mathcal{S}$ disjoint to a given generator.
	\par The third and fourth condition for pseudopolar sets hold for strong pseudopolar sets thanks to Lemma \ref{aantallen}.\comments{The third condition of kind sets is a direct consequence of condition (i) of strong pseudopolar sets, applied for $i=1$ and $\pi\notin\mathcal{S}$. The fourth condition of kind sets is a straightforward weakening of condition (iii) of strong pseudopolar sets.}
\end{proof}

\comments{Based on the previous theorem, we present two different approaches to prove the main result. We will prove that both nice sets and kind sets are embedded polar spaces. In both proofs the four axioms of a polar space are verified. In both cases for three of them the verification is easy. The differences between nice and kind sets lead to different verifications of the last axiom, which is non-trivial to check.}

We now prove that pseudopolar sets are embedded polar spaces. Thanks to the previous lemma this is sufficient to prove the main theorem.

\begin{theorem}\label{niceembedded}
	Let $\mathcal{P}$ be a finite classical polar space of rank $d\geq3$ and with parameter $e\geq1$, embedded in a projective space over $\F_{q}$. If $\mathcal{S}$ is a pseudopolar set of generators in $\mathcal{P}$, then $\mathcal{S}$ is the set of generators of a classical polar space of rank $d$ and with parameter $e-1$ embedded in $\mathcal{P}$.
\end{theorem}
\begin{proof}
	Let $\mathcal{O}$ be the set of all points in $\mathcal{P}$ that are contained in at least one element of $\mathcal{S}$ and let $\mathcal{T}$ be the set of all subspaces in $\mathcal{P}$ that are contained in at least one element of $\mathcal{S}$. It is clear that the sets $\mathcal{O}$ and $\mathcal{S}$ are subsets of $\mathcal{T}$, as is the empty set. We define $\mathcal{P}'=(\mathcal{O},\mathcal{T})$. We will prove that $\mathcal{P}'$ is a polar space of rank $d$ using Definition \ref{defpolarspace}.
	It is immediate that $\mathcal{P}'$ fulfils axioms (1) and (2): they are inherited from $\mathcal{P}$. Note that the elements of $\mathcal{T}$ that have dimension $d-1$ are precisely the elements of $\mathcal{S}$. By condition (ii) of pseudopolar sets we know that the number of elements of $\mathcal{S}$ disjoint to a given element of $\mathcal{S}$ is nonzero. Hence, axiom (4) is also fulfilled.
	\par We count the number of points in $\mathcal{O}$. We perform a double counting on the set $\{(P,\pi)\mid P\in\mathcal{O},\pi\in\mathcal{S},P\in\pi \}$. Using conditions (iii) and (iv) of pseudopolar sets we find that
	\[
	|\mathcal{O}|=\frac{\gs{d}{1}{q}\prod_{i=0}^{d-1}(q^{e+i-1}+1)}{\prod_{i=0}^{d-2}(q^{e+i-1}+1)}=\gs{d}{1}{q}(q^{e+d-2}+1)\;.
	\]
	\par Now, let $\pi\in\mathcal{S}$ be a generator. We count the number of tuples in the set $E=\{(Q,\tau)\mid Q\in\mathcal{O}\setminus\pi,\tau\in\mathcal{S},\dim(\pi\cap\tau)=d-2,Q\in\tau\}$. Using condition (i) of pseudopolar sets, we find that
	\[
	|E|=\gs{d}{1}{q}q^{e-1}\cdot q^{d-1}=\gs{d}{1}{q}q^{d+e-2}
	\]
	So, $|E|=|\mathcal{O}\setminus\pi|$. Consequently, for the points in $\mathcal{O}$ not in $\pi$ there is on average one generator in $\mathcal{S}$ through it meeting $\pi$ in a $(d-2)$-space. However, by axiom (3) applied for $\mathcal{P}$, for every point in $\mathcal{O}$ there is at most one generator in $\mathcal{S}$ through it meeting $\pi$ in a $(d-2)$-space. We find that through every point of $\mathcal{O}$ there is precisely one generator in $\mathcal{S}$ meeting $\pi$ in a $(d-2)$-space. This proves axiom (3).
	\par We conclude that $\mathcal{P}'$ is a polar space of rank $d$. Comparing Theorem \ref{subspacesonpolar} and $|\mathcal{O}|$ we find immediately that the parameter of $\mathcal{P}'$ equals $e-1$. By the aforementioned result of Tits (see \cite{tits}) the polar space $\mathcal{P}'$ is classical.
\end{proof}

As a corollary we find the main result of this article, the one that motivated this research.

\begin{corollary}
	Let $\mathcal{P}$ be a finite classical polar space of rank $d\geq3$ and with parameter $e\geq1$, embedded in a projective space over $\F_{q}$. If $\mathcal{S}$ is a strong pseudopolar set of generators in $\mathcal{P}$, then $\mathcal{S}$ is the set of generators of a classical polar space of rank $d$ and with parameter $e-1$ embedded in $\mathcal{P}$.
\end{corollary}
\begin{proof}
	Considering Lemma \ref{goodnicekind} this corollary follows from Theorem \ref{niceembedded}\comments{ as well as from Theorem \ref{kindembedded}}.
\end{proof}

Looking at the previous results we can see that actually we proved a characterisation result for embedded polar spaces of the same rank and with parameter one less.

\begin{theorem}
	Let $\mathcal{P}$ be a finite classical polar space of rank $d\geq3$ and with parameter $e\geq1$, embedded in a projective space over $\F_{q}$, and let $\mathcal{S}$ be a set of generators in $\mathcal{P}$. The four following statements are equivalent.
	\begin{itemize}
		\item[(1)] $\mathcal{S}$ is strong pseudopolar.
		\item[(2)] $\mathcal{S}$ is pseudopolar.
		%\item[(3)] $\mathcal{S}$ is kind.
		\item[(3)] $\mathcal{S}$ is the set of generators of a classical polar space of rank $d$ and with parameter $e-1$ embedded in $\mathcal{P}$.
	\end{itemize}
\end{theorem}
\begin{proof}
	The assertion $(1)\implies(2)$ \comments{and $(1)\implies(3)$ are }is in Lemma \ref{goodnicekind} and the assertion $(2)\implies(3)$ is Theorem \ref{niceembedded}\comments{, while the assertion $(3)\implies(4)$ is Theorem \ref{kindembedded}}. We conclude this proof by showing the assertion $(3)\implies (1)$.
	\par We assume that (3) is true. Let $\mathcal{P}'$ be the polar space defined by $\mathcal{S}$. Both $\mathcal{P}$ and $\mathcal{P}'$ are classical. So, we have one of the following possibilities:
	\begin{center}
		\centering
		\begin{tabular}{c||c|c|c|c}
			$\mathcal{P}$ & $\mathcal{Q}(2d,q)$ & $\mathcal{W}(2d-1,q)$, $q$ even & $\mathcal{H}(2d,q^{2})$ & $\mathcal{Q}^{-}(2d+1,q)$ \\\hline
			$\mathcal{P}'$ & $\mathcal{Q}^{+}(2d-1,q)$ & $\mathcal{Q}^{+}(2d-1,q)$, $q$ even & $\mathcal{H}(2d-1,q^{2})$ & $\mathcal{Q}(2d,q)$
		\end{tabular}
	\end{center}
	The second case, the embedding of $\mathcal{Q}^{+}(2d-1,q)$ in $\mathcal{W}(2d-1,q)$ for $q$ even, arises from the embedding of $\mathcal{Q}^{+}(2d-1,q)$ in $\mathcal{Q}(2d,q)$, $q$ even, since $\mathcal{Q}(2d,q)$ and $\mathcal{W}(2d-1,q)$ are isomorphic as polar space for $q$ even (we refer to \cite[Chapter 22]{Hirschfeld} and \cite[Chapter 11]{tay} for more details). So, we only have to deal with the three remaining cases. In all these cases, the embedded polar space $\mathcal{P}'$ can be found by intersecting $\mathcal{P}$ with a non-tangent hyperplane $\alpha$ of the ambient projective space. The generators of $\mathcal{P}'$, i.e. the set $\mathcal{S}$, are the generators of $\mathcal{P}$ contained in $\alpha$. The generators of $\mathcal{P}$ that are not in $\mathcal{S}$ meet $\alpha$ in a $(d-2)$-space. In the same way, the points of $\mathcal{P}'$ are the points of $\mathcal{P}$ contained in $\alpha$, while the points of $\mathcal{P}$ that are not in $\mathcal{P}'$ are the points not in $\alpha$. Keeping this in mind, checking conditions (i) and (iii) of strong pseudopolar sets are well-known counting results for finite classical polar spaces (see e.g.~\cite{bcn} or \cite{Hirschfeld}). Condition (ii) follows from Lemma \ref{subspacesonpolar} as $\prod^{d-1}_{i=0}(q^{e+i}+1)>\prod^{d-2}_{i=0}(q^{e+i}+1)$.
\end{proof}

\begin{remark}
	It is clear, as we saw in Lemma \ref{goodnicekind} that the definition of pseudopolar \comments{and kind }sets are weakened versions of the definition of strong pseudopolar sets. However, we can also prove the stronger conditions from the weaker ones. One could wonder whether there are other ways to weaken the conditions of strong pseudopolar sets and also find an equivalent definition. \par This can indeed be done. Probably there are a lot of possibilities; we will present here one (without the proofs). One could keep conditions (i) and (ii) of pseudopolar sets (see Definition \ref{defnice}) and replace (iii) and (iv) by
	\begin{itemize}
		\item[(iii')] the number of elements of $\mathcal{S}$ meeting a given generator not in $\mathcal{S}$ in a $(d-2)$-space equals $q^{e-1}+1$;
		\item[(iv')] for every point $P$ of $\mathcal{P}$ and every generator $\pi\notin\mathcal{S}$ through $P$, either for all $j=0,\dots,d-2$ the number of generators of $\mathcal{S}$ through $P$ meeting $\tau$ in a $j$-space is non-zero and it equals $q^{e+1}+1$ if $j=d-2$, or else for all $j=0,\dots,d-2$ it is zero.
	\end{itemize}
\end{remark}

\begin{remark}
	We could consider the set of generators of a polar space of rank $d$ as a subset of the set of $(d-1)$-spaces of the ambient projective space instead of as generators of larger polar space. We could wonder whether it is possible to characterise a set of $(d-1)$-spaces in a projective space as the set of generators in a way similar to the one described in Definitions \ref{defgood} or \ref{defnice}\comments{ or \ref{defkind}}. For strong pseudopolar \comments{and kind }sets it is easy to see that an analogue cannot be made as it requires a statement about the $(d-1)$-spaces not in the set, and we know that not all $(d-1)$-spaces meet the polar space in the same way.
	\par For pseudopolar sets, however, it is possible to make a projective analogue. For given $d$ and $e$, we could call a set $\mathcal{S}$ of $(d-1)$-spaces in a projective space of dimension at least $2d+e-1$ pseudopolar if  
	\begin{itemize}
		\item[(i)] the number of elements of $\mathcal{S}$ meeting a given element of $\mathcal{S}$ in a $(d-2)$-space equals $\gs{d}{1}{q}q^{e-1}$;
		\item[(ii)] the number of elements of $\mathcal{S}$ disjoint to a given element of $\mathcal{S}$ is nonzero;
		\item[(iii)] $|\mathcal{S}|=\prod_{i=0}^{d-1}(q^{e+i-1}+1)$;
		\item[(iv)] any point in the projective space is contained in $0$ or $\prod_{i=0}^{d-2}(q^{e+i-1}+1)$ elements of $\mathcal{S}$.
	\end{itemize}
	The generator set of a polar space of rank $d$ and with parameter $e$ is clearly an example of a pseudopolar set, but we will give another example.
	\par Let $\pi_{1}$ and $\pi_{2}$ be two planes in $\PG(n,q)$, $n\geq4$ and $q$ even, meeting in a line $\ell$, and let $\mathcal{H}_{i}$ be a dual hyperoval in $\pi_{i}$ with $\ell$ as one of its $q+2$ lines, $i=1,2$. Then $(\mathcal{H}_{1}\setminus\{\ell\})\cup(\mathcal{H}_{2}\setminus\{\ell\})$ is a pseudopolar set for $d=2$ and $e=1$. It is however not the set of generators of a polar space. Note that the `classical' example is a $\mathcal{Q}^{+}(3,q)$, and that is spans also a 3-space in $\PG(n,q)$. 
	\par A classification of pseudopolar sets in projective spaces is therefore different from the classification for polar spaces. We state it here as an open problem.
\end{remark}

\paragraph*{Acknowledgement:} The research of Maarten De Boeck is supported by the BOF--UGent (Special Research Fund of Ghent University).

\end{document}